\pdfoutput=1
\documentclass{amsart}
\usepackage{graphicx}
\newtheorem{ex}{Example}[section]
\newtheorem{theorem}{Theorem}[section]
\newtheorem{lemma}{Lemma}[section]
\newtheorem{proposition}{Proposition}[section]
\newtheorem{remark}{Remark}[section]

\begin{document}

\markboth{N. Ito}
{Jones polynomials of long virtual knots}

\title{Jones polynomials of long virtual knots}

\author{NOBORU ITO}

\address{Department of Mathematics, Waseda University, 3-4-1 Okubo, Shinjuku-ku, Tokyo 169-8555, Japan\\
noboru@moegi.waseda.jp}

\begin{abstract}
This paper defines versions of the Jones polynomial and Khovanov homology by using several maps from the set of Gauss diagrams to its variant.  Through calculation of some examples, this paper also shows that these versions behave differently from the original ones.
\end{abstract}
\maketitle

\section{Introduction}\label{intro}
Knots, which are circles embedded in thickened surfaces, are often treated as the virtual knots introduced by Kauffman \cite{kauffman}.  Virtual knots with base points are regarded as long virtual knots, since a circle less one point is homeomorphic to a line.  

On the other hand, a Gauss diagram is a circle with chords, where the preimages of each double point of the immersion are connected by the chords.  Virtual knots are nothing but equivalence classes of Gauss diagrams.  We can place some information on the circle and chords of a Gauss diagram.  

This paper considers maps between Gauss diagrams, and it is possible to produce some versions of a single knot invariant.  In particular, there is a simple way to define invariants for long virtual knots thorough Gauss diagrams.  In this paper, we consider versions of the Jones polynomial in terms of invariants of long virtual knots.  We also see that this approach is effective for Khovanov homology.

The plan of this paper is as follows: Sec. \ref{functor_nano} gives a precise definition of long virtual knots and the corresponding Gauss diagrams.  Sec. \ref{ver_jones} obtains definitions of the maps between Gauss diagrams and defines versions of the Jones polynomial.  We see in Sec. \ref{app_kh} that the same approach is good for Khovanov homology.  

\section{Long virtual knots and their presentations as Gauss diagrams}\label{functor_nano}
Virtual knot theory was introduced by Kauffman \cite{kauffman} and virtual knots are often treated as Gauss diagrams.  
\subsection{Knots, knot diagrams, long knots, long knot diagrams, and Gauss diagrams}
A {\it{knot}} is a circle smoothly embedding into $\mathbb{R}^{3}$ and a {\it{long knot}} is a smooth embedding $\mathbb{R}$ $\to$ $\mathbb{R}^{3}$.  These are often represented by {\it{knot diagrams}} or {\it{long knot diagrams}}, which are images of generic immersions of the circle into the plane adding the information on overpasses and underpasses at double points, as shown in Figs. \ref{knot_and_long} (a) and (b).  A long knot is often identified as a knot with a point, called a {{base point}}, on the circle.  Its diagram is presented as a knot diagram with a base point on curves distinct from the double points (Fig. \ref{knot_and_long} (c)).  
\begin{figure}
\begin{picture}(0,0)
\put(60,10){(a)}
\put(160,10){(b)}
\put(295,10){(c)}
\end{picture}
\includegraphics[width=12cm]{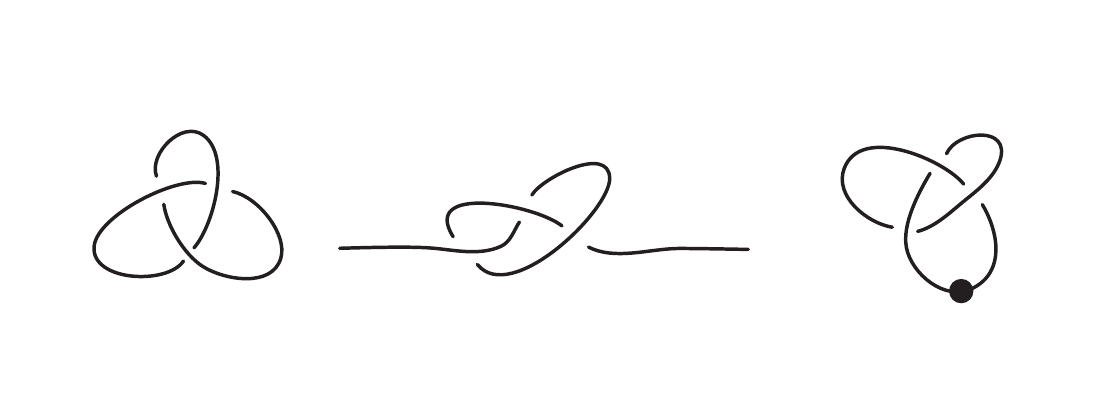}
\caption{(a) Knot diagram.  (b) Long knot diagram.  (c) Knot diagram with base point.}\label{knot_and_long}
\end{figure}
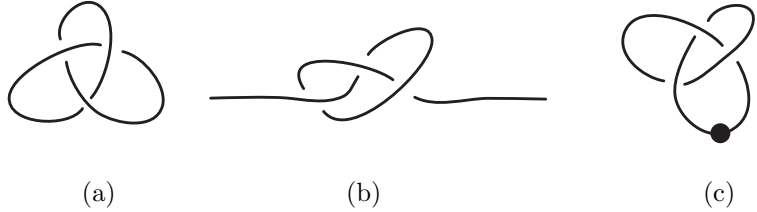
In this paper, we treat knot diagrams with finite double points only.  As is well known, two knots are isotopic knots if related by a finite sequence of {\it{Reidemeister moves}}, which are local moves on knot diagrams as shown in Fig. \ref{reidemeister_m}.  
\begin{figure}
\begin{picture}(0,0)
\put(50,38){$\Omega_1a$}
\put(104,38){$\Omega_1b$}
\put(195,37){$\Omega_2$}
\put(292,36){$\Omega_3$}
\end{picture}
\includegraphics[width=12cm]{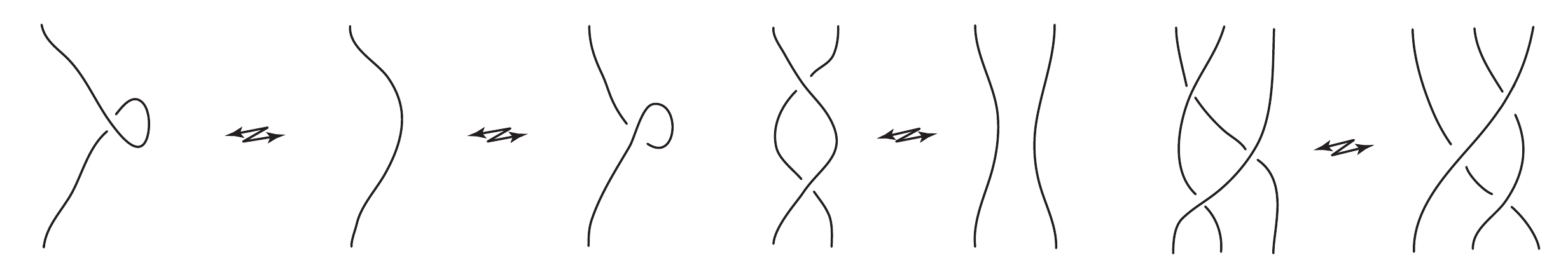}
\caption{Reidemeister moves.  The local replacements on the neighborhoods are drawn, and the exteriors of the neighborhoods are the same for both diagrams of each move.}\label{reidemeister_m}
\end{figure}
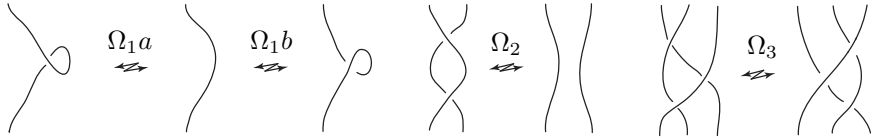
If necessary, we add an adjective such as {\it{classical}} for referring to the knots defined above and keep this role for other objects: long knots, knot diagrams, and long knot diagrams.  

Every generic immersion of a circle into the plane fixes a {\it{Gauss diagram}} that is a circle with chords, where the preimages of each double point of the immersion are connected by the chords (Fig. \ref{diag_to_gauss}).  {\it{Oriented Gauss diagrams}} are considered up to orientation preserving homeomorphism underlying circles, and the orientations imply those of knots.  In this paper, the underlying circle of every oriented Gauss diagram has counterclockwise orientation.  In the rest of this paper, unless otherwise specified, we adopt oriented Gauss diagrams that are simply called Gauss diagrams.  To recover a knot up to isotopy from a Gauss diagram, we ascribe signs and arrows for every chord.  The sign of a chord is defined as the local writhe number of the corresponding double point, and the arrow of a chord is oriented from the upper branch to the lower branch.  
\begin{figure}
\includegraphics[width=13cm]{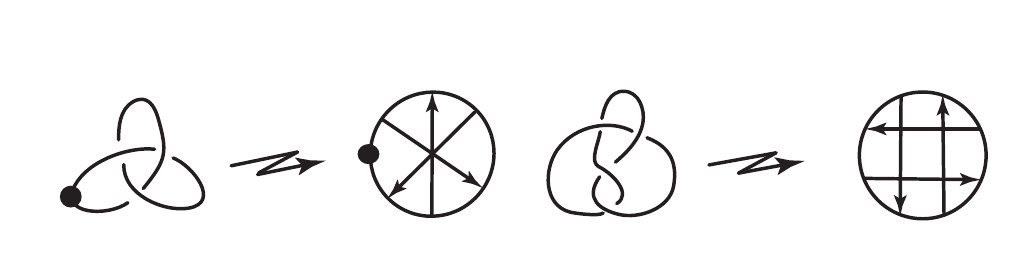}
\begin{picture}(0,0)
\put(-40,65){\tiny$+$}
\put(-12,60){\tiny$+$}
\put(-32,34){\tiny$+$}
\put(146,67){\tiny$-$}
\put(168,53){\tiny$+$}
\put(153,30){\tiny$-$}
\put(131,45){\tiny$+$}
\end{picture}
\caption{A long knot diagram and a knot diagram are encoded by Gauss diagrams.  }\label{diag_to_gauss}
\end{figure}
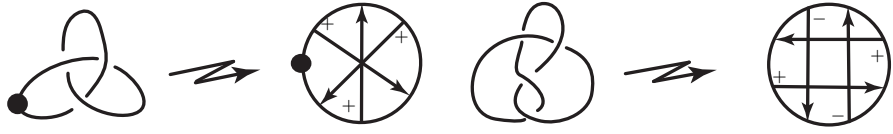 
In the same way, we define Gauss diagrams of long knot diagrams as in Fig. \ref{diag_to_gauss}.  
\subsection{Virtual knots, virtual knot diagrams, long virtual knots, and long virtual knot diagrams}\label{virtual_sec}
A {\it{virtual knot}}, introduced by Kauffman \cite{kauffman}, is defined as follows: A {\it{virtual knot diagram}} is a smooth immersion of the circle into the plane such that all singular points are transversal double points.  These double points are divided into real crossing points and virtual crossing points, where real crossing points have information on overpasses and underpasses as for the classical knot diagrams shown in Fig. \ref{virtual_crossing}.  
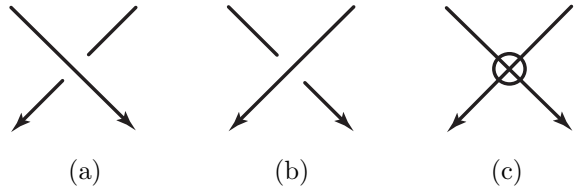
\begin{figure}
\begin{center}
\begin{picture}(0,0)
\put(60,0){(a)}
\put(138,0){(b)}
\put(220,0){(c)}
\end{picture}
\includegraphics[width=10cm]{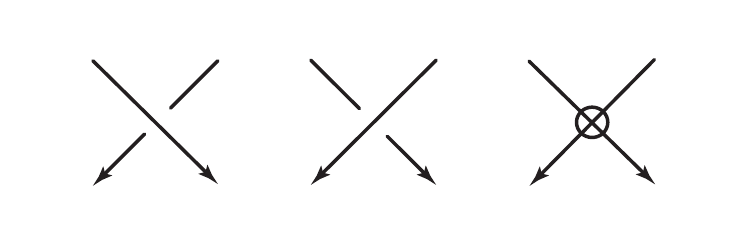}
\caption{(a), (b): Real crossings.  (c): Virtual crossing.  }\label{virtual_crossing}
\end{center}
\end{figure}
A branch consisting of a virtual crossing is not divided into an overpass and an underpass.  {\it{Virtual knots}} are the set of virtual knot diagrams divided by Reidemeister moves and the {\it{virtual moves}} shown in Fig. \ref{virtual_moves}.  
\begin{figure}
\begin{center}
\includegraphics[width=12cm]{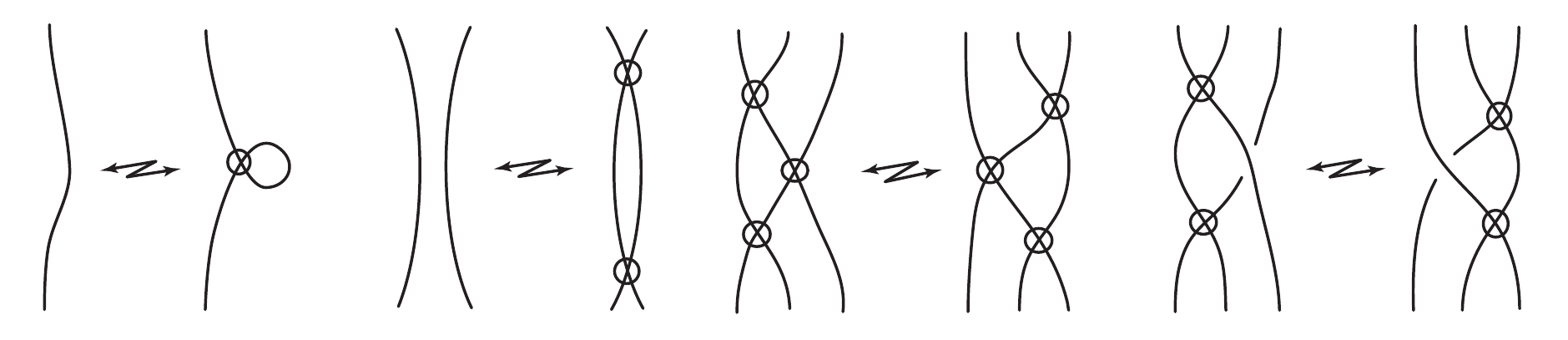}
\caption{Virtual moves.}\label{virtual_moves}
\end{center}
\end{figure}
For virtual knots, the following fact was proved by Goussarov, Polyak, and Viro \cite{gpv} using group systems:
\begin{theorem}[Goussarov, Polyak, Viro]
Virtually isotopic classical knots are isotopic.  
\end{theorem}

Here, we enhance the definition of knot diagrams and long knot diagrams for treating virtual knots as classical knots following works by Carter, Kamada, and Saito \cite{cks} and N. Kamada and S. Kamada \cite{kk} (see also Kauffman \cite{kauffman} and Goussarov, Polyak, and Viro \cite{gpv}).  In the rest of this paper, objects such as knots or knot diagrams (i.e., containing classical knots, virtual knots, classical long knots, or long virtual knots) are regarded as oriented, unless confusion is likely to occur.  {\it{Knot diagrams on surfaces}} are images of generic immersions of the circle into an oriented surface adding information on overpasses and underpasses at double points.  {\it{Long knot diagrams on surfaces}} are knot diagrams on surfaces with base point on curves distinct from the double points.  As is well known, {\it{virtual knots}} (resp. {\it{long virtual knots}}) are {\it{stable equivalence}} classes of knot diagrams (resp. long virtual knot diagrams) on surfaces.  The definition of the stable equivalence is as follows: Two knot diagrams on surfaces that are images of generic immersions are stably equivalent if they can be replaced by a finite sequence of {\it{stable homeomorphisms}} and Reidemeister moves in the ambient surfaces.  Two images of generic immersions are stably homeomorphic if there is a homeomorphism of their regular neighborhoods in the ambient surfaces that maps the first diagram onto the second one and preserves the overcrossings and undercrossings as well as the orientations of the surface and the immersed curve.  Two long knot diagrams on surfaces are stably equivalent if they can be replaced by a finite sequence of stable homeomorphisms preserving the base point and Reidemeister moves in the ambient surfaces away from the base point.  In particular, we now have a purely combinatorial proof that there are injective maps from classical knots (resp. long knots) to virtual knots (resp. long virtual knots) (cf. Turaev \cite{turaev2}).

\subsection{Gauss diagrams for virtual knots and long virtual knots.}
Gauss diagrams of virtual knots and long virtual knots are defined by knot diagrams and long knot diagrams on surfaces in the same way as for classical knot diagrams (resp. classical long knot diagrams) that are generic immersions of circles (resp. circles with base points) into the plane.  The alternative definition of Gauss diagrams of virtual knots and long virtual knots is that Gauss diagrams are constructed by using virtual knot diagrams and long virtual knot diagrams on the plane in the same way as for classical knot diagrams, but all virtual crossings are disregarded as shown in Fig. \ref{virtual_gauss}.  
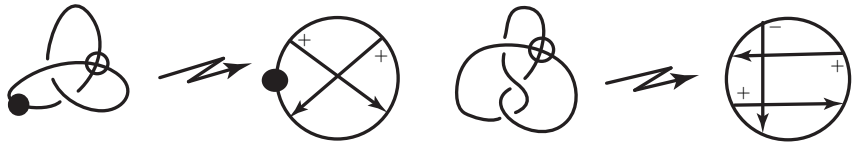
\begin{figure}
\includegraphics[width=13cm]{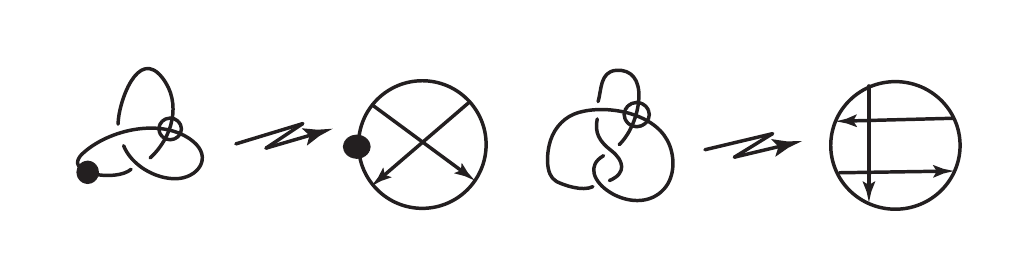}
\begin{picture}(0,0)
\put(-43,68){\tiny$+$}
\put(-14,62){\tiny$+$}
\put(158.5,58){\tiny$+$}
\put(123,48){\tiny$+$}
\put(135,73){\tiny$-$}
\end{picture}
\caption{Gauss diagrams for a long virtual knot and a virtual knot.}\label{virtual_gauss}
\end{figure}
Here, the following important fact \cite[Theorem 1.A]{gpv} should be mentioned: 
\begin{theorem}[Goussarov, Polyak, Viro]\label{gpv_thm}
A Gauss diagram defines a virtual knot diagram up to virtual moves.  
\end{theorem}
Then, a virtual knot (resp. long virtual knot) equals to the corresponding Gauss diagram (resp. Gauss diagram with a base point) considered up to moves that are the counterparts of Reidemeister moves for Gauss diagrams (resp. Gauss diagrams with base points) as shown in Fig. \ref{rel_gauss_a}.  
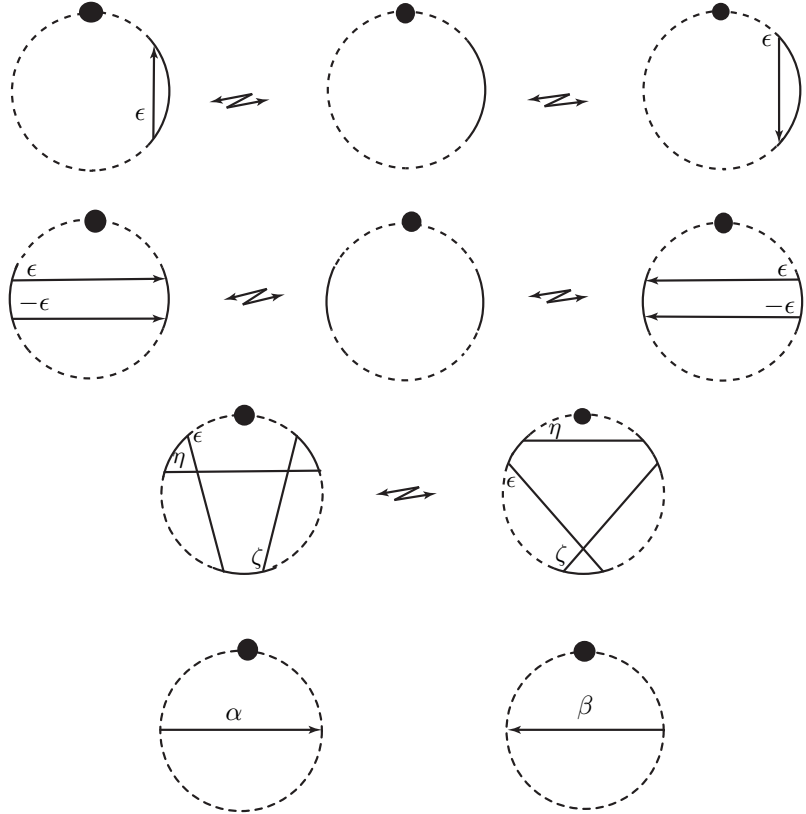
\begin{figure}
\begin{picture}(0,0)
\put(76,272){$\epsilon$}
\put(313,300){$\epsilon$}
\put(32,200){$- \epsilon$}
\put(35,213){$\epsilon$}
\put(314,199){$- \epsilon$}
\put(319,213){$\epsilon$}
\put(235,105){\small$\zeta$}
\put(232.5,153.3){\small$\eta$}
\put(216.5,133){\small$\epsilon$}
\put(90.3,142){\small$\eta$}
\put(98,150){\small$\epsilon$}
\put(120,104){\small$\zeta$}
\end{picture}
\begin{picture}(0,0)
\put(107,45){$\alpha$}
\put(240,47){$\beta$}
\end{picture}
\includegraphics[width=12cm]{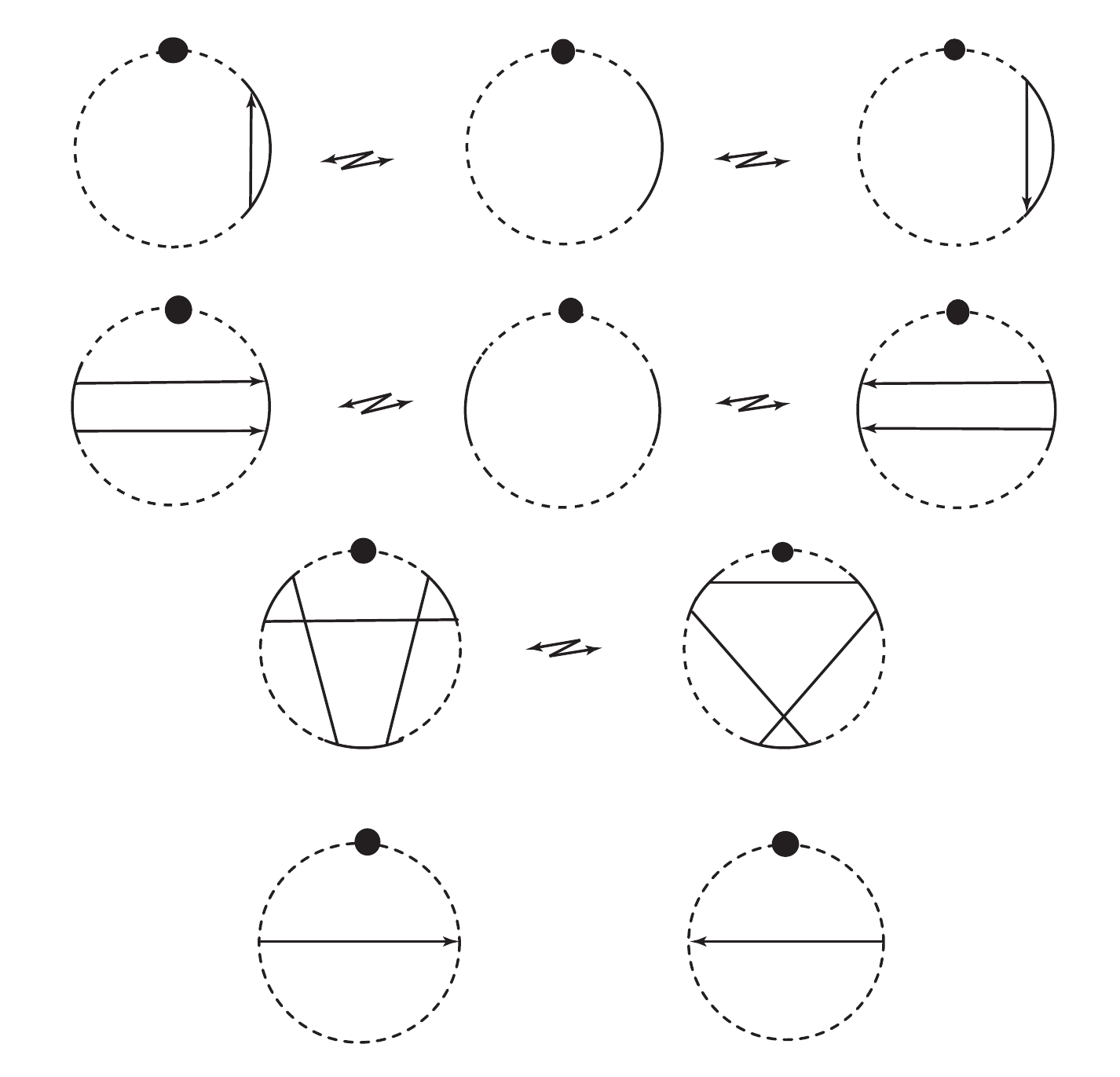}
\caption{Relations of Gauss diagrams (resp. Gauss diagrams with base points) corresponding to Reidemeister moves of virtual knots (resp. long virtual knots) where $\epsilon$, $\eta$, and $\zeta$ are $+$ or $-$, but $(\epsilon, \eta, \zeta)$ is $(\pm, \pm, \pm)$, $(\mp, \mp, \pm)$, or $(\mp, \pm, \pm)$ in the third row.  Directions of chords in the third row, denoted by $\alpha$ and $\beta$ in the fourth row, are defined by Table \ref{table_direction}.  }\label{rel_gauss_a}
\end{figure}
\begin{table}
\begin{center}
{\setlength{\tabcolsep}{10pt}
\begin{tabular}{ccc} \hline
Case & signs & arrows \\ \hline
1 & $(+, +, +)$ & $(\alpha, \alpha, \alpha)$ \\ \hline
2 & $(+, +, +)$ & $(\beta, \beta, \beta)$ \\ \hline
3 & $(-, -, -)$ & $(\alpha, \alpha, \alpha)$ \\ \hline
4 & $(-, -, -)$ & $(\beta, \beta, \beta)$ \\ \hline \hline
5 & $(+, +, -)$ & $(\alpha, \alpha, \beta)$ \\ \hline
6 & $(+, +, -)$ & $(\beta, \beta, \alpha)$ \\ \hline
7 & $(-, -, +)$ & $(\alpha, \alpha, \beta)$ \\ \hline
8 & $(-, -, +)$ & $(\beta, \beta, \alpha)$ \\ \hline \hline
9 & $(+, -, -)$ & $(\alpha, \beta, \beta)$ \\ \hline
10 & $(+, -, -)$ & $(\beta, \alpha, \alpha)$ \\ \hline
11 & $(-, +, +)$ & $(\alpha, \beta, \beta)$ \\ \hline
12 & $(-, +, +)$ & $(\beta, \alpha, \alpha)$ \\ \hline 
\end{tabular}
}
\end{center}
\caption{Rules for the triples of three chords in the third row of Fig. \ref{rel_gauss_a}.  Double lines indicate that we can regard these twelve cases as three groups.}\label{table_direction}
\end{table}
\begin{remark}
A Gauss diagram naturally has the orientation of a circle.  Hence, if we adopt the notion of Gauss diagrams for non-oriented knots, Gauss diagrams should be identified up to given arbitrary orientations.  On the other hand, when we consider an oriented Gauss diagram, the order of trivalent vertices on the Gauss diagram is fixed.  That is why, in this paper, we represent Reidemeister move $\Omega_3$ as the third line of Fig. \ref{rel_gauss_a}.  Using \cite{turaev2}, we have the following.  
\end{remark}
\begin{lemma}\label{reide_negative}
A long virtual knot is generated by Fig. \ref{rel_gauss_a}.  A virtual knot is generated by Fig. \ref{rel_gauss_a} neglecting the base points.  
\end{lemma}

\section{Versions of the Jones polynomial}\label{ver_jones}
In this section, the Gauss diagrams are oriented Gauss diagrams and have relations corresponding to Reidemeister moves.  The symbol $\epsilon$ stands for $+$ or $-$ as in Fig. \ref{rel_gauss_a}.  

First, let us consider Gauss diagrams neglecting the directions of arrows on chords.  Then, the map $p_r$ is defined by correspondences of codes: 
\begin{picture}(0,20)
\put(10,7){\circle{20}}
\put(27,5){$\mapsto$}
\put(55,7){\circle{20}}
\put(3,0){\vector(1,1){14}}
\put(2,4.5){$\epsilon$}
\put(48,0){\line(1,1){14}}
\put(46,4){$\epsilon$}
\put(65,0){.}
\end{picture}

The projection $p_{r}$ induces relations on the set of Gauss diagrams neglecting the directions of arrows.  This topology is determined by Fig. \ref{rel_gauss_a} except for neglecting the directions of arrows.  The topological objects are called pseudolinks (resp. long pseudolinks) for virtual knots (resp. long virtual knots).  

Turaev obtained the following fact \cite[Section 8.3]{turaev2} through his nanoword theory: 
\begin{theorem}[Turaev]\label{main}
The Jones polynomial $V_{K}$ of an oriented knot $K$ is defined by $p_r(G)$, where $G$ is a Gauss diagram of $K$; i.e., $V_{K}$ $=$ $V_{p_r(G)}$.  
\end{theorem}

Second, we consider the map $p$ from Gauss diagrams with base points to Gauss diagrams neglecting signs of arrows on chords as follows: 

\begin{picture}(0,20)
\put(2,13){\circle*{3}}
\put(10,7){\circle{20}}
\put(13,6.5){\tiny$-$}
\put(17,14){\vector(-1,-1){14}}
\put(20,0){,}
\end{picture}
\qquad
\begin{picture}(0,20)
\put(2,13){\circle*{3}}
\put(10,7){\circle{20}}
\put(27,5){$\mapsto$}
\put(47,13){\circle*{3}}
\put(55,7){\circle{20}}
\put(3,0){\vector(1,1){14}}
\put(1,4.5){\tiny$+$}
\put(46.5,4.5){\small$a$}
\put(48,0){\line(1,1){14}}
\put(65,0){,}
\put(75,0){\text{and}}
\end{picture}
\qquad
\begin{picture}(0,20)
\put(79,13){\circle*{3}}
\put(87,7){\circle{20}}
\put(78,4.5){\tiny$-$}
\put(80,0){\vector(1,1){14}}
\put(97,0){,}
\end{picture}
\begin{picture}(0,20)
\put(102,13){\circle*{3}}
\put(110,7){\circle{20}}
\put(127,5){$\mapsto$}
\put(147,13){\circle*{3}}
\put(155,7){\circle{20}}
\put(113,6.5){\tiny$+$}
\put(117,14){\vector(-1,-1){14}}
\put(148,4.5){\small$b$}
\put(162,14){\line(-1,-1){14}}
\put(165,0){.}
\end{picture}

The projection $p$ means the underlying curves, called open flat virtual knots, for long virtual knots.  This topology is determined by the relations of the Gauss diagrams with base points in Fig. \ref{rel_gauss_a} where Table \ref{table_direction} is restricted to Cases 1 and 3, except for replacing $+$ (resp. $-$) with $a$ (resp. $b$) and neglecting the directions of arrows.  

Third, we consider the map $i$ between Gauss diagrams as follows: 

\begin{picture}(0,20)
\put(2,13){\circle*{3}}
\put(10,7){\circle{20}}
\put(27,5){$\mapsto$}
\put(47,13){\circle*{3}}
\put(55,7){\circle{20}}
\put(3,0){\line(1,1){14}}
\put(1,4.5){\small$a$}
\put(46.5,4.5){\tiny$+$}
\put(48,0){\line(1,1){14}}
\put(65,0){,}
\put(102,13){\circle*{3}}
\put(75,0){\text{and}}
\put(110,7){\circle{20}}
\put(127,5){$\mapsto$}
\put(147,13){\circle*{3}}
\put(155,7){\circle{20}}
\put(103,4.5){\small$b$}
\put(117,14){\line(-1,-1){14}}
\put(145.5,4.5){\tiny$-$}
\put(162,14){\line(-1,-1){14}}
\put(165,0){.}
\end{picture}

\begin{theorem}\label{main_s0}
Let $D$ be a diagram of an arbitrary long virtual knot $K$.  The map $V_{i(p(D))}$ is an invariant of the long virtual knot $K$.  
\end{theorem}
\begin{proof}
The map $i$ sends open flat virtual knots to long pseudolinks.  The map is well defined, since the relations of open flat virtual knots corresponding to Fig. \ref{rel_gauss_a} are sent to the relation defined by the same Gauss diagrams with $a$ (resp. $b$) replacing $+$ (resp. $-$) while neglecting arrow directions.  By replacing $p_r(G)$ of Theorem \ref{main} with $i(p(K))$, another Jones polynomial $V_{i(p(K))}$ becomes an invariant of an arbitrary long virtual knot $K$.  
\end{proof}

\begin{remark}
Fukunaga regarded the map $i$ as the one producing a topological invariant \cite{fukunaga}.  
\end{remark}

Here, in order to capture the graphical meaning of the map $i \circ p$, we prove Theorem \ref{main_s0} in another way as below.  
\begin{proof}
Let $K$ be a long virtual knot and $D_K$ its diagram on a surface (cf. Sec. \ref{virtual_sec}).  We can consider the map $p$ to mean that every crossing of $D_K$ is replaced with a transversal double point.  Without loss of generality, we can assume by invoking plane isotopy that every crossing consists of two orthogonal branches.  Hence, we assume this condition in the rest of the proof.  Under this assumption, the definition of $p$ is represented as
\begin{equation}\label{p-eq}
\begin{split}
\begin{picture}(35,35)
\put(0,11){$p : $}
\put(20,0){\line(1,1){30}}
\put(20,30){\line(1,-1){10}}
\put(50,0){\line(-1,1){10}}
\put(60,11){$\mapsto$}
\put(80,0){\line(1,1){30}}
\put(80,30){\line(1,-1){30}}
\end{picture}
\qquad\qquad
\end{split}
\end{equation}
for a sufficiently small neighborhood of every crossing, where the exterior of the neighborhoods of the crossings is mapped to itself and contains the base point.  Then, by $p$, the curve $p(D_K)$ with the base point on a surface is determined to stable homeomorphisms preserving the base point and orientations of the curve and the surface.  Every transversal double point has exactly two tangent vectors $t_1$ and $t_2$, so there exist two types of crossings: one type has a positively oriented pair $(t_1, t_2)$ and the other has a negatively oriented pair $(t_1, t_2)$.  

More graphically, if the ambient surface containing the curve has counterclockwise orientation, every double point of $p(D_K)$ belongs to exactly one of two types: 
\begin{equation*}
\begin{picture}(50,50)
\put(0,5){$1$st}
\put(32,5){$2$nd}
\put(10,15){\vector(1,1){30}}
\put(40,15){\vector(-1,1){30}}
\put(60,5){,}
\put(104,5){$1$st}
\put(72,5){$2$nd}
\put(80,15){\vector(1,1){30}}
\put(110,15){\vector(-1,1){30}}
\end{picture}
\qquad\qquad
\end{equation*}
where $1$st (resp. $2$nd) means the first (resp. second) branch passing trough the double point starting from the base point.

Without loss of generality, we can assume that the ambient surface containing $D_K$ or $p(D_K)$ has counterclockwise orientation in the rest of the proof.  Under this assumption, for these two types of double points, we consider the map $q$ as follows: 

\begin{equation}\label{q-eq}
\begin{split}
\begin{picture}(130,50)
\put(0,0){$1$st}
\put(30,0){$2$nd}
\put(10,10){\vector(1,1){30}}
\put(40,10){\vector(-1,1){30}}
\put(60,20){$\mapsto$}
\put(110,10){\line(-1,1){10}}
\put(90,30){\vector(-1,1){10}}
\put(80,10){\vector(1,1){30}}
\put(120,10){,}
\end{picture}
\begin{picture}(100,50)
\put(0,0){$2$nd}
\put(33,0){$1$st}
\put(10,10){\vector(1,1){30}}
\put(40,10){\vector(-1,1){30}}
\put(60,20){$\mapsto$}
\put(110,10){\vector(-1,1){30}}
\put(100,30){\vector(1,1){10}}
\put(80,10){\line(1,1){10}}
\end{picture}
\end{split}
\end{equation}
for a sufficiently small neighborhood of every double point, where the exterior of the neighborhoods of the double points is mapped to itself and contains the base point.  The image $q \circ p(D_K)$ becomes a long virtual knot diagram.  

In what follows, we show that if $D_{K_1}$ and $D_{K_2}$ are stably equivalent, $q \circ p(D_{K_1})$ and $q \circ p(D_{K_2})$ are stably equivalent.  

According to the definition of $q \circ p$ by (\ref{p-eq}) and (\ref{q-eq}), if $D_{K_1}$ and $D_{K_2}$ are stably homeomorphic, preserving the base point and the orientations of the curve and the surface, so are $q \circ p(D_{K_1})$ and $q \circ p(D_{K_2})$.  Subsequently, we will verify that if $D_{K_1}$ and $D_{K_2}$ can be replaced by Reidemeister moves in the ambient surface away from the base point, so can $q \circ p(D_{K_1})$ and $q \circ p(D_{K_2})$.  

\begin{itemize}
\item Reidemeister moves $\Omega_1 a$ and $\Omega_1 b$.  

Let $D_1$ (resp. $D_2$) be the local diagram defined by the left (resp. right) side of the move $\Omega_1 a$ in Fig. \ref{reidemeister_m}, and let $D_3$ be the local diagram defined by the right side of the move $\Omega_1 b$ in Fig. \ref{reidemeister_m}.  For each of $D_1$, $D_2$, and $D_3$, there are two cases by choice of orientation.  If the orientation of $D$ $=$ $D_1$, $D_2$, or $D_3$ is along the direction from the bottom to the top (resp. from the top to the bottom), we denote the local diagram by $D^{u}$ (resp. $D^{d}$) where $u$ (resp. $d$) stands for up (resp. down).  Then, we have to check the following four pairs: $(D_3^{u}, D_2^{u})$ (Case 1), $(D_1^{u}, D_2^{u})$ (Case 2), $(D_1^{d}, D_2^{d})$ (Case 3), and $(D_3^{d}, D_2^{d})$ (Case 4).  Since each check is similar to the others, we first show the one for Case 2.  

According to the definition of $q \circ p$ by (\ref{p-eq}) and (\ref{q-eq}), $q \circ p(D_1^{u})$ $=$ $D_3^{u}$.  On the other hand, $q \circ p(D_2^{u})$ $=$ $D_2^{u}$.  Since $D_3^{u}$ and $D_2^{u}$ can be replaced by Reidemeister move $\Omega_1 b$, so can $q \circ p(D_1^{u})$ and $q \circ p(D_2^{u})$.  

Using the list below, we can show the other cases by analogy.  

Case 1: $q \circ p(D_3^{u})$ $=$ $D_3^{u}$.  

Case 2: $q \circ p(D_1^{u})$ $=$ $D_3^{u}$.  

Case 3: $q \circ p(D_1^{d})$ $=$ $D_1^{d}$.  

Case 4: $q \circ p(D_3^{d})$ $=$ $D_1^{d}$.  

\item Reidemeister move $\Omega_2$.  

Let $D_1$ (resp. $D_2$) be the local diagram defined by the left (resp. right) side of the move $\Omega_2$ in Fig. \ref{reidemeister_m}.  
For $D$ $=$ $D_1$ or $D_2$, let $D_r$ be the local diagram obtained by looking at $D$ upside down as shown in Fig. \ref{diagramref}.  
\begin{figure}[h!]
\begin{center}
\includegraphics[width=4cm]{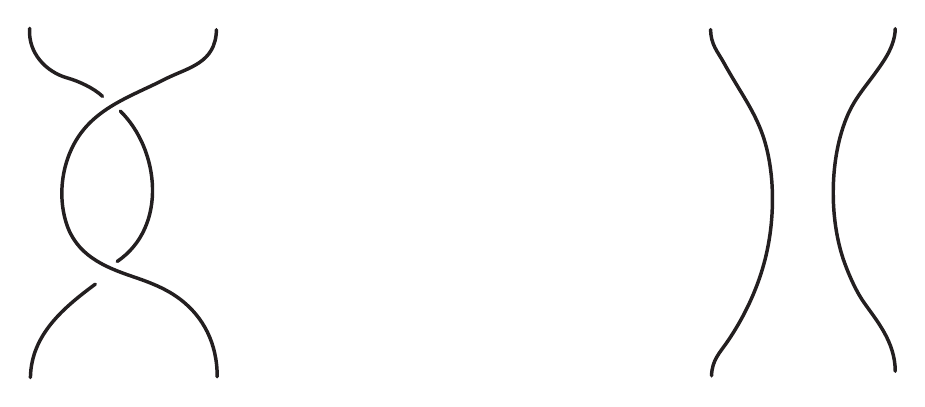}
\caption{The local diagrams $D_{1r}$ (left) and $D_{2r}$ (right).}\label{diagramref}
\end{center}
\end{figure}
By definition, $D_{2r}$ is the same as $D_{2}$.

For each of $D_1$, $D_2$, $D_{1r}$, and $D_{2r}$, there are four cases by choice of orientation.  If the orientations of the two branches of $D$ $=$ $D_1$ or $D_2$ are both in the direction from the bottom to the top (resp. from the top to the bottom), we denote the local diagram by $D^{uu}$ (resp. $D^{dd}$).  Similarly, $D^{ud}$ (resp. $D^{du}$) stands for the local diagram $D$ where the orientations of the two branches are upward (resp. downward) and downward (resp. upward) from the left.  Now, by Lemma \ref{reide_negative}, it is sufficient here to consider only the cases of $D^{ud}$ and $D^{du}$.  

The local diagram $D$ $=$ $D_1$, $D_2$, $D_{1r}$, or $D_{2r}$ consists of two branches.  The branch in which the endpoints are at the bottom left and the top left is called the left branch, and the other is called the right branch.  If the first branch of $D^{ud}$ is the right (resp. the left) when starting from the base point, we denote the local diagram by $D^{\overline{ud}}$ (resp. $D^{ud}$).  If the first branch of $D^{du}$ is the right (resp. the left), we denote the local diagram by $D^{\overline{du}}$ (resp. $D^{du}$).  There are some relations between the oriented $D$ and $D_{r}$ that can be observed by looking at these upside down.  For example, when we look at $D_{1r}^{\overline{ud}}$ upside down, we see $D_1^{ud}$.  We can recognize ``looking at it upside down'' as the operator $f_{\pi}$, and using this operator we have 
\begin{equation}\label{pi_formula}
\begin{split}
f_{\pi}(D_{1r}^{ud}) = D_{1}^{\overline{ud}}~{\rm{and}}~f_{\pi}(D_{2r}^{ud}) = D_{2}^{\overline{ud}},\\
f_{\pi}(D_{1r}^{du}) = D_{1}^{\overline{du}}~{\rm{and}}~f_{\pi}(D_{2r}^{du}) = D_{2}^{\overline{du}},\\
f_{\pi}(D_{1r}^{\overline{ud}}) = D_{1}^{ud}~{\rm{and}}~f_{\pi}(D_{2r}^{\overline{ud}}) = D_{2}^{ud},\\
f_{\pi}(D_{1r}^{\overline{du}}) = D_{1}^{du}~{\rm{and}}~f_{\pi}(D_{2r}^{\overline{du}}) = D_{2}^{du}.  
\end{split}
\end{equation}
In the eight formulae of (\ref{pi_formula}), $f_{\pi}$ behaves as the involution.  

The second row of Fig. \ref{rel_gauss_a} shows the eight moves between $D_{1}$ and $D_{2}$ or between $D_{1r}$ and $D_{2r}$ as follows ($\ast$ $=$ $1$ or $2$): $D_\ast^{ud}$ (Case 1), $D_\ast^{\overline{ud}}$ (Case 2), $D_\ast^{du}$ (Case 3), $D_\ast^{\overline{du}}$ (Case 4), $D_{\ast r}^{ud}$ (Case 5), $D_{\ast r}^{\overline{ud}}$ (Case 6), $D_{\ast r}^{du}$ (Case 7), and $D_{\ast r}^{\overline{du}}$ (Case 8).  We would like to show that the move between $q \circ p(D_1)$ and $q \circ p(D_2)$ is one of these eight cases.  However, if (\ref{pi_formula}) is used, it is sufficient to check only Cases 1 -- 4.  

Since each check is similar to the others, we first show the one for Case 2.  According to the definition of $q \circ p$ by (\ref{p-eq}) and (\ref{q-eq}), $q \circ p(D_1^{\overline{ud}})$ $=$ $D_{1r}^{\overline{ud}}$.  Likewise, $q \circ p(D_2^{\overline{ud}})$ $=$ $D_2^{\overline{ud}}$ $=$ $D_{2r}^{\overline{ud}}$.  Therefore, $q \circ p(D_1^{\overline{ud}})$ and $q \circ p(D_2^{\overline{ud}})$ can be replaced by the Reidemeister move of Case 6.  

Using the list below, we can show the other cases by analogy.  

Case 1: $q \circ p(D_1^{ud})$ $=$ $D_1^{ud}$.  

Case 2: $q \circ p(D_1^{\overline{ud}})$ $=$ $D_{1r}^{\overline{ud}}$.  

Case 3: $q \circ p(D_1^{du})$ $=$ $D_1^{du}$.  

Case 4: $q \circ p(D_{1}^{\overline{du}})$ $=$ $D_{1r}^{\overline{du}}$.  

\item Reidemeister moves similar to $\Omega_3$.  

Let us recall that an equivalence relation for a long virtual knot is defined by Lemma \ref{reide_negative} and Fig. \ref{rel_gauss_a}.  We have already verified the invariance of $V_{q \circ p(K)}$ under the moves in the first and second rows of Fig. \ref{rel_gauss_a}.  Consequently, it is sufficient to show the invariance of $V_{q \circ p(K)}$ under the moves in the third row of Fig. \ref{rel_gauss_a}.  

The moves in the third row of Fig. \ref{rel_gauss_a} are explained by Table \ref{table_direction}, which is realized as Fig. \ref{3rd_move_ver} by using the local knot diagrams.  
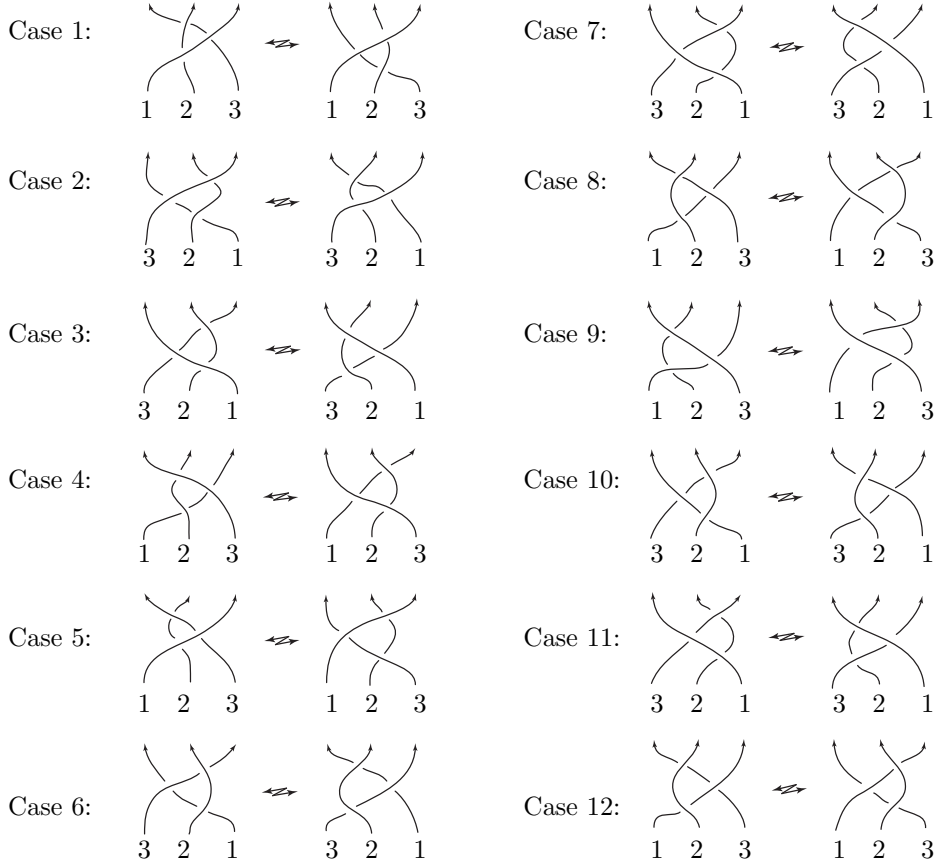
\begin{figure}
\qquad\qquad
\includegraphics[width=11cm]{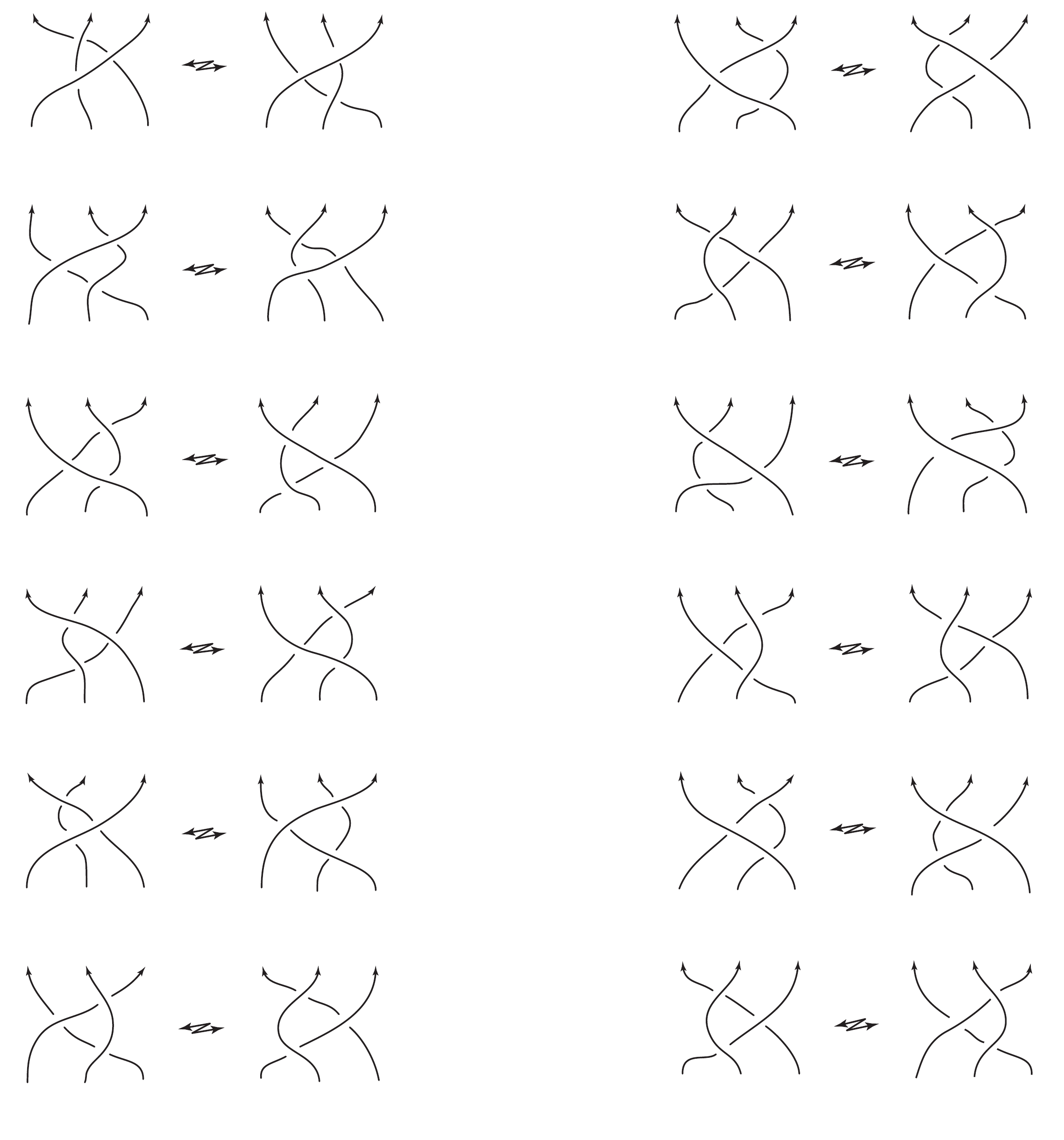}

\begin{picture}(0,0)
\put(-180,36){Case 6: }
\put(15,36){Case 12: }
\put(-180,100){Case 5: }
\put(15,100){Case 11: }
\put(-180,160){Case 4: }
\put(15,160){Case 10: }
\put(-180,215){Case 3: }
\put(15,215){Case 9: }
\put(-180,273){Case 2: }
\put(15,273){Case 8: }
\put(-180,330){Case 1: }
\put(15,330){Case 7: }
\put(-131,20){$3$\quad $2$ \quad$1$}
\put(-131,75){$1$\quad $2$ \quad$3$}
\put(-131,132){$1$\quad $2$ \quad$3$}
\put(-131,186){$3$\quad $2$ \quad $1$}
\put(-129,244){$3$\quad $2$ \quad $1$}
\put(-130,300){$1$\quad $2$ \quad $3$}
\put(-60,20){$3$\quad $2$ \quad$1$}
\put(-60,75){$1$\quad $2$ \quad$3$}
\put(-60,132){$1$\quad $2$ \quad$3$}
\put(-60,186){$3$\quad $2$ \quad $1$}
\put(-60,244){$3$\quad $2$ \quad $1$}
\put(-60,300){$1$\quad $2$ \quad $3$}
\put(63,20){$1$\quad $2$ \quad $3$}
\put(63,75){$3$\quad $2$ \quad $1$}
\put(63,132){$3$\quad $2$ \quad $1$}
\put(63,186){$1$\quad $2$ \quad $3$}
\put(63,244){$1$\quad $2$ \quad $3$}
\put(63,300){$3$\quad $2$ \quad $1$}
\put(132,20){$1$\quad $2$ \quad $3$}
\put(132,75){$3$\quad $2$ \quad $1$}
\put(132,132){$3$\quad $2$ \quad $1$}
\put(132,186){$1$\quad $2$ \quad $3$}
\put(132,244){$1$\quad $2$ \quad $3$}
\put(132,300){$3$\quad $2$ \quad $1$}
\end{picture}
\caption{Reidemeister moves that should be checked.  These cases correspond to Table \ref{table_direction}.  Numbers 1--3 indicate the order of branches, which is defined as the order for passing through the neighborhood when starting from the base point.}\label{3rd_move_ver}
\end{figure}

Let $D_{il}$ (resp. $D_{ir}$) be the local diagram defined by the left (resp. right) side of the move in Case $i$ of Fig. \ref{3rd_move_ver}.  According to the definition of $q \circ p$ by (\ref{p-eq}) and (\ref{q-eq}), if $i$ $=$ $1$, $4$, $5$, $8$, $9$, or $12$, we have $q \circ p(D_{il})$ $=$ $D_{2l}$ and $q \circ p(D_{ir})$ $=$ $D_{2r}$.  Similarly, if $i$ $=$ $2$, $3$, $6$, $7$, $10$, or $11$, we have $q \circ p(D_{il})$ $=$ $D_{8l}$ and $q \circ p(D_{ir})$ $=$ $D_{8r}$.

Here, we denote one of the Reidemeister moves between $D_{il}$ and $D_{ir}$ ($1 \le i \le 12$) by $\sim$, and we have 
\begin{equation}\label{3rd-move-eq}
\begin{split}
q \circ p(D_{il}) &= D_{2l} \sim D_{2r} = q \circ p(D_{ir}) \qquad (i = 1, 4, 5, 8, 9, 12), \\
q \circ p(D_{il}) &= D_{8l} \sim D_{8r} = q \circ p(D_{ir}) \qquad (i = 2, 3, 6, 7, 10, 11).  
\end{split}
\end{equation}
\end{itemize}
The formulae (\ref{3rd-move-eq}) complete the check that $q \circ p(D_{il})$ and $q \circ p(D_{ir})$ can be replaced by one of the Reidemeister moves between $D_{il}$ and $D_{ir}$ ($1 \le i \le 12$).  

As proved above, map $q \circ p$ is well defined as the map from the set of long virtual knots to itself.  

On the other hand, we can assume that the domain of the map $p_r$ is the set of long virtual knots.  Under this assumption, Theorem \ref{main} implies that $V_{K}$ $=$ $V_{{p_r}(K)}$.  Here, we notice that $p_r \circ q$ $=$ $i$.  Then, we have
\[V_{i \circ p(K)} = V_{p_r \circ q \circ p(K)} = V_{q \circ p(K)}.\]
Therefore, the map $V_{i \circ p(K)}$ is well defined as the map from the set of long virtual knots.  That is, the map $V_{i \circ p(K)}$ is an invariant for Reidemeister moves and virtual moves.  This completes the proof.
\end{proof}

In what follows, we show applications of Theorem \ref{main_s0}.  

\begin{ex}\label{ex_phrase}  
Let $K_1$ and $K_2$ be the long virtual knots shown in Fig. \ref{k1_k2}, with Jones polynomials $V_{K_1}(t)$ $=$ $V_{K_2}(t)$ $=$ $t^{-2}$ $+$ $t^{- 3/2}$ $-$ $t^{-1}$ $-$ $t^{- 1/2}$ $+$ $t^{1/2}$.  However, $V_{i(p(K_1))}(t)$ $=$ $V_{K_1}(t)$ $\neq$ $V_{K_1}(t^{-1})$ $=$ $V_{K_2}(t^{-1})$ $=$ $V_{i(p(K_2))}(t)$.  
\begin{figure}
\begin{center}
\includegraphics[width=12cm]{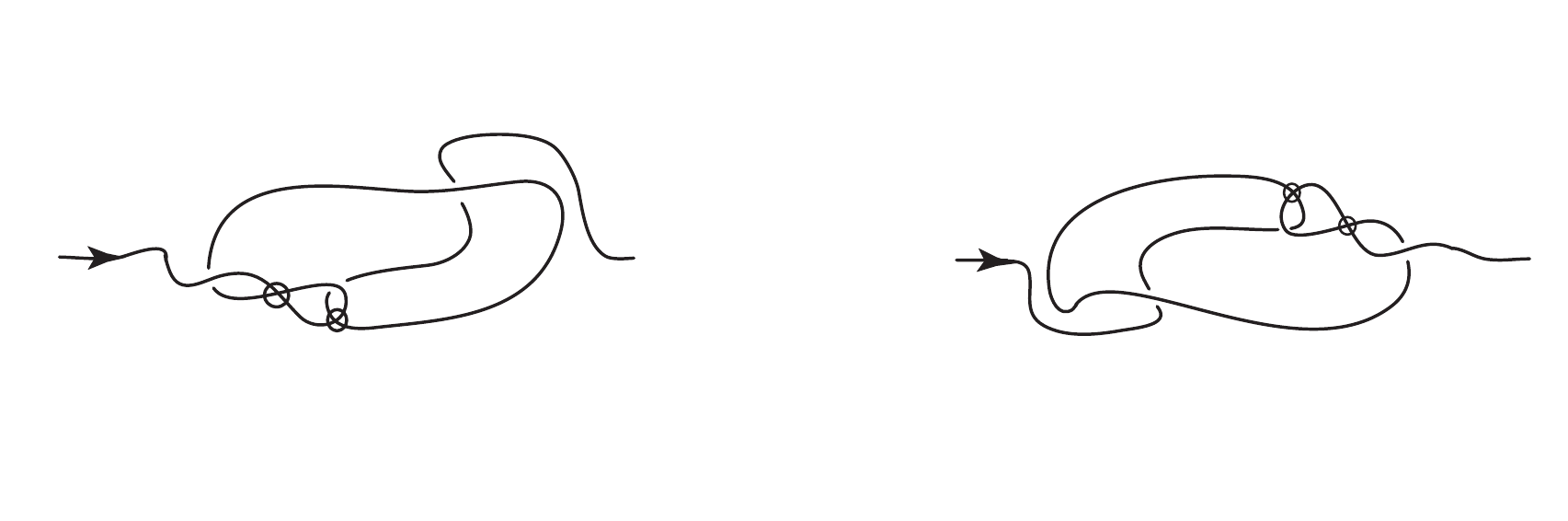}
\caption{Two long virtual knots $K_1$ (left) and $K_2$ (right).}\label{k1_k2}
\end{center}
\end{figure}
\end{ex}
Example \ref{ex_phrase} implies the following: 
\begin{theorem}\label{strongerthan}
Let $K$ be a long virtual knot $K$.  For $K$, the pair of $V_{K}$ and $V_{i(p(K))}$ is a stronger invariant than the polynomial $V_K$.  In other words, there exist two long virtual knots $K_1$ and $K_2$ such that $V_{K_1}$ $=$ $V_{K_2}$ but $V_{i(p(K_1))}$ $\neq$ $V_{i(p(K_2))}$.  
\end{theorem}
\begin{proof}
The above example demonstrates the statement.  
\end{proof}
Example \ref{ex_phrase} also means that $V_{K}$ detects the orientation of the long virtual knot for $K_1$.  Let $-K$ be a knot with an orientation that is the reverse of that for a knot $K$.  
\begin{remark}\label{ori}
Let $K$ be a long virtual knot that has a Gauss diagram which satisfies the condition that when arrow directions are neglected, the Gauss diagram is symmetric with respect to a line passing thorough the base point (e.g. the right figure of Fig. \ref{virtual_gauss}).  If a knot $K$ satisfies the assumption ($\diamond$) that the Jones polynomial $V_{i(p(K))}$ changes by replacing $t^{1/2}$ $\mapsto$ $t^{- 1/2}$, then the polynomial $V_{i(p(K))}$ of $K$ detects the orientation of $K$ (e.g., $K$ $=$ $K_1$).  This is because of the well-known fact that the Jones polynomial $V_{\overline{K}}$ of the mirror image $\overline{K}$ is obtained by replacing $t^{1/2}$ with $t^{- 1/2}$.  In other words, by the assumption ($\diamond$), $V_{i(p(K))}$ $\neq$ $V_{i(p(\overline{K}))}$ $=$ $V_{i(p(-K))}$.  However, there is no example satisfying the assumption ($\diamond$) for classical long knots since an arbitrary open flat virtual knot on the plane is equal to the trivial open flat virtual knot under its relations.  Here, the consideration is summarized by the following proposition: 
\end{remark}
\begin{proposition}\label{ori_prop}
Let $K$ be an arbitrary nonclassical long virtual knot and $\overline{K}$ its mirror image.  If the Jones polynomial $V_{K}(t)$ is not symmetric with respect to $t^{1/2} \mapsto t^{-1/2}$, the Jones polynomial $V_{i(p(K))}(t)$ detects its orientation.  
\end{proposition}

Next, we consider another type of example.  
\begin{ex}\label{ex2_phrase}
Let $K_3$ and $K_4$ be the long virtual knots shown in Fig. \ref{k3_k4}.  Then, $V_{K_3}$ $=$ $V_{K_4}$ but $V_{i(p(K_3))}(t)$ $=$ $V_{K_3}(t^{-1})$ $\neq$ $V_{K_3}(t)$ $=$ $V_{K_4}(t)$ $=$ $V_{i(p(K_4))}(t)$.  
\begin{figure}
\begin{center}
\includegraphics[width=10cm]{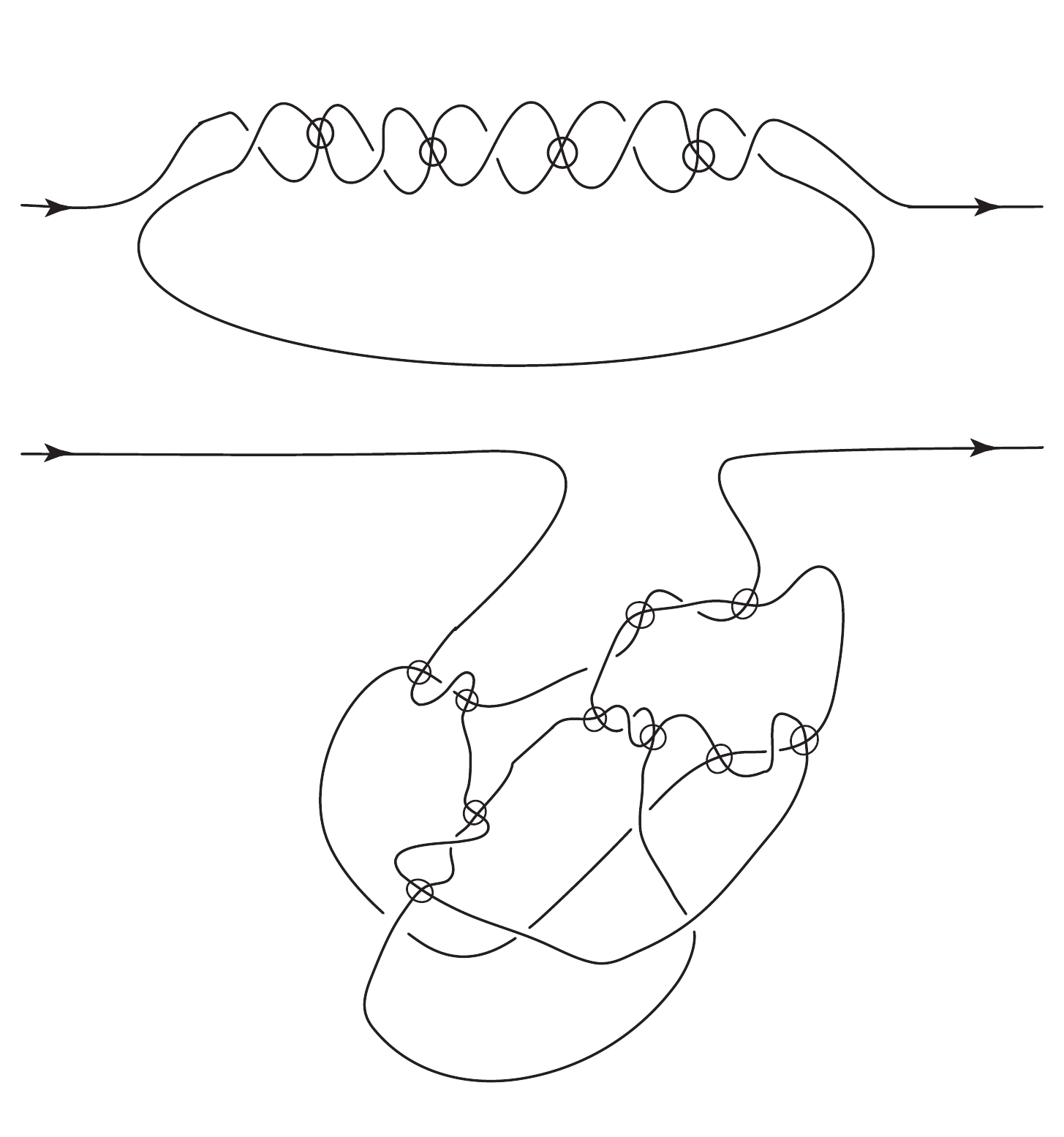}
\caption{Two long virtual knots $K_3$ (upper) and $K_4$ (lower).}\label{k3_k4}
\end{center}
\end{figure}
\end{ex}

Similarly, $p_a$ is defined as the composition $\tau_0 \circ p$ of the two maps $p$ and the involution map $\tau_0 :$ $a$ $\mapsto$ $b$ on chords of Gauss diagrams with base points.  Moreover, $p_{ra}$ is defined as the composition $\tau_1 \circ p_r$ of two maps $p_r$ and the involution $\tau_1 :$ $+$ $\mapsto$ $-$ on chords of Gauss diagrams (we consider Gauss diagrams with base points if necessary).  We can also consider the map $i_a$ defined as the composition $\tau_1 \circ i$.  It is easy to see that these maps imply well-defined maps between equivalence classes of Gauss diagrams determined by topological objects which we treated.  Then, we have the following.  

\begin{theorem}\label{four_thm}
All the choices of $p_r$, $p_{ra}$, $p$, $p_a$, $i$, and $i_a$, together generate four types of the Jones polynomials for long virtual knots.  

As a corollary to Theorem \ref{strongerthan}, the tuple of four versions of the Jones polynomial is stronger than the Jones polynomial for long virtual knots.  
\end{theorem}
\begin{proof}
Considering every combination of $p_r$, $p_{ra}$, $p$, $p_a$, $i$, and $i_a$ yields the formulas $i \circ p_a$ $=$ $i_a \circ p$ and $i_a \circ p_a$ $=$ $i \circ p$.  
\end{proof}
We consider these maps in Examples \ref{ex_pra} and \ref{ia}.  

Here, let us consider the graphical meaning of these variations in the $V_K$ of a long virtual knot $K$.  Recall the definition of $q \circ p$ by (\ref{p-eq}) and (\ref{q-eq}).  For the diagram $D$ of a long virtual knot, $q \circ p_a (D)$ $=$ $q \circ \tau_0 \circ p(D)$ is the mirror image of $q \circ p(D)$.  On the other hand, when we denote the mirror image of $D$ by $D^{*}$, we have $V_{p_{ra} (D)}$ $=$ $V_{D^{*}}$.  The pair of Jones polynomials $(V_{p_r (D)}, V_{p_{ra} (D)}, V_{i \circ p}(D), V_{i_a \circ p}(D))$ is nothing but $(V_{D}, V_{D^{*}}, V_{q \circ p(D)}, V_{\left(q \circ p(D)\right)^{*}})$; that is, we calculate the four values of the Jones polynomials of two diagrams of long virtual knots and their mirror images.

\section{Application of the discussion to Khovanov homology}\label{app_kh}
In this section, we apply the above discussion to Khovanov homology.  After we recall the Khovanov homology, we consider the above discussion for Khovanov homology in the case of the coefficient $\mathbb{Z}_2$.  
\subsection{Khovanov homology}
In this section, we recall the Khovanov homology of the Jones polynomial introduced by Khovanov \cite{khovanovjones}.  There are two major redefinitions of Khovanov homology (\cite{bar-natan}, \cite{viro}), and here we give a brief review of the definition in the style of Viro \cite{viro}.  

For a given knot diagram, let us consider a small edge, called a {\it{marker}}, for each crossing on the link diagram.  In the rest of this paper, we can use a simple notation such as that of Fig. \ref{markers} (c) for the marker of Fig. \ref{markers} (a).  Every marker has its sign defined as in Fig. \ref{markers}.  The signed markers determine the directions of smoothing for all crossings (Fig. \ref{smoothing}).  The smoothened link diagram is called the {\it{Kauffman state}} or simply the {\it{state}}.  
\begin{figure}
\begin{center}
\includegraphics[width=8cm]{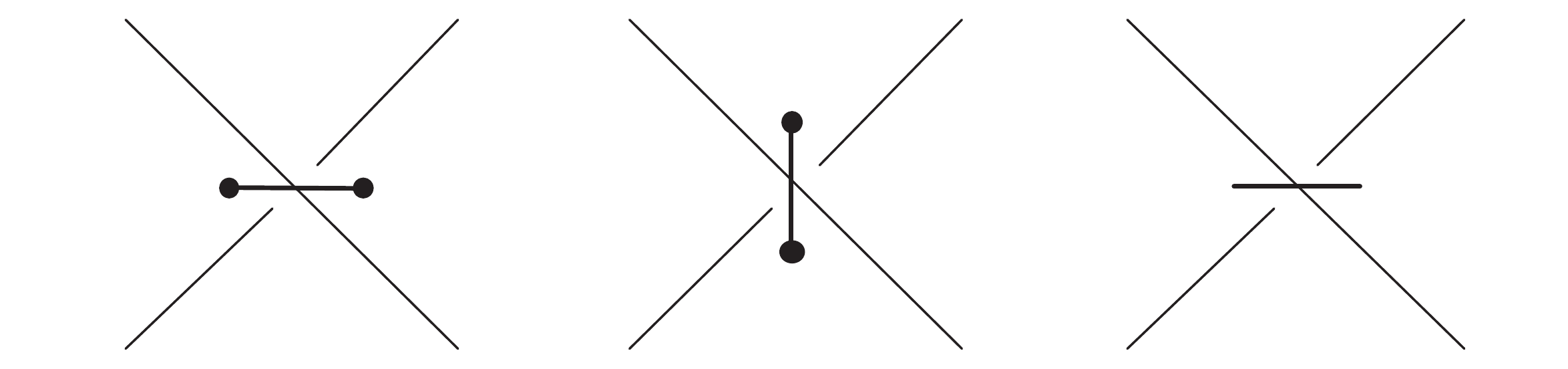}
\end{center}
\begin{picture}(90,0)
\put(-35,5){(a)}
\put(37,5){(b)}
\put(111,5){(c)}
\end{picture}
\caption{(a) Positive marker.  (b) Negative marker.  (c) Simple notation for positive marker.}\label{markers}
\end{figure}
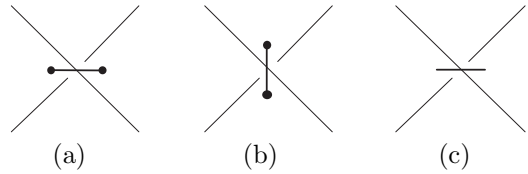
\begin{figure}
\begin{center}
\begin{picture}(0,90)
\put(0,30){(b)}
\put(0,115){(a)}
\end{picture}\qquad\qquad\qquad
\includegraphics[width=7cm]{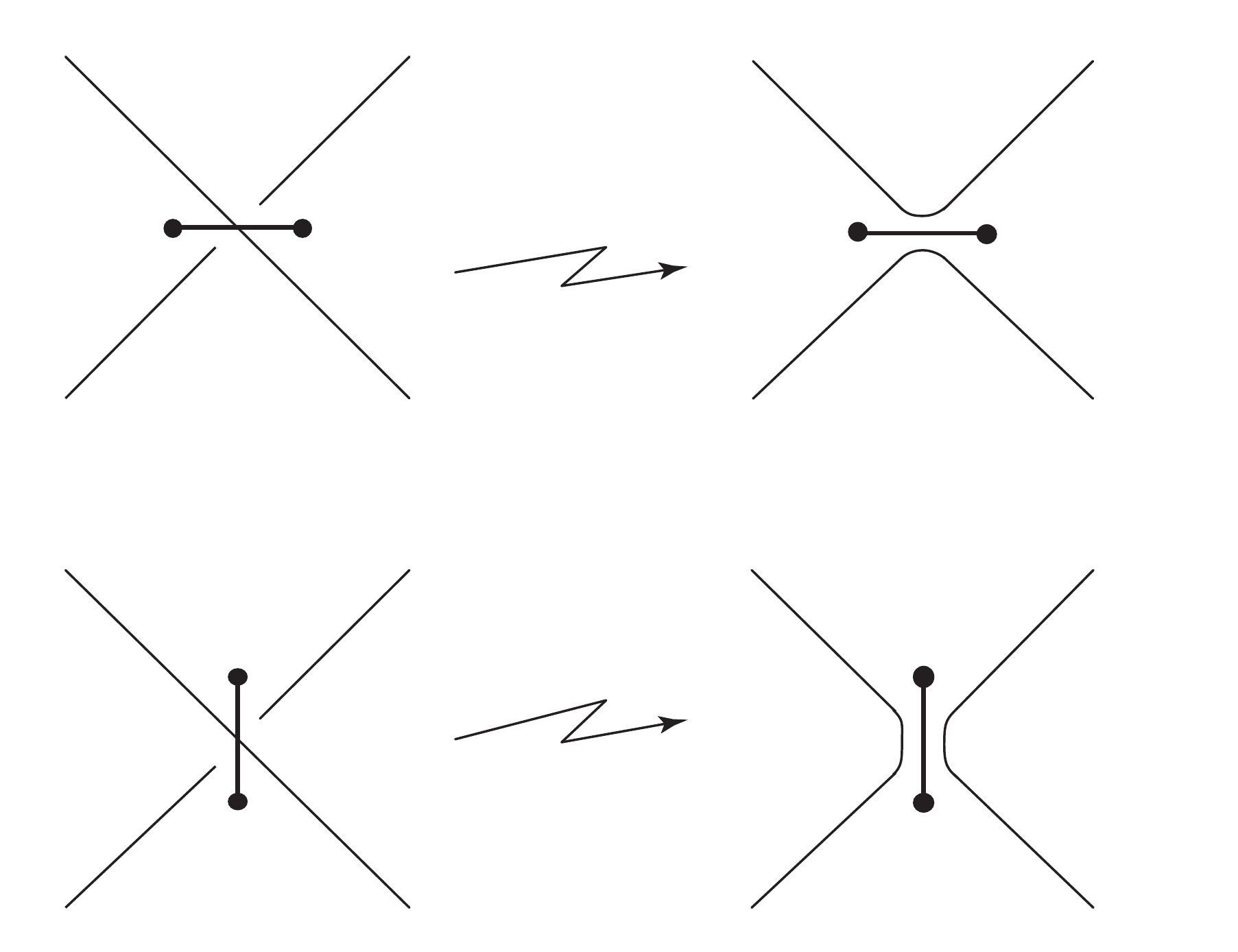}
\caption{Smoothing producing states.  The marker on the crossing in (a) is the positive marker, and that in (b) is the negative marker.}\label{smoothing}
\end{center}
\end{figure}
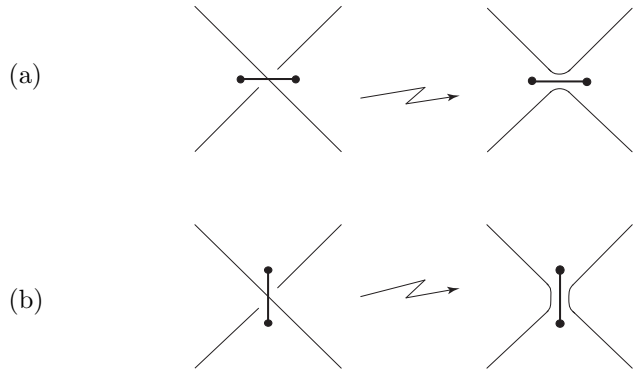
In the next step, we assign {\it{labels}} $x$ or $1$ for every circle of the state.  The {\it{degree}} of $y$ $=$ $x$ or $y$ $=$ $1$ is defined by ${\rm{deg}}(x)$ $=$ $-1$ or ${\rm{deg}}(1)$ $=$ $1$.  The state whose circles have labels $x$ or $1$ is called an {\it{enhanced state}} and is denoted by $S$.  Let $\sigma(S)$ be the number of positive markers minus the number of negative markers for an arbitrary enhanced state $S$.  For a label $y$ $=$ $x$ or $y$ $=$ $1$, we set $\tau(S)$ $=$ $\sum_{{\text{circles}}~y~{\text{in}}~S}$ ${\rm{deg}}(y)$.  For a link diagram $D$ of a link $L$, the unnormalized Jones polynomial $\hat{J}(L)$ of a link $L$ is obtained as
\begin{equation}
\hat{J}(L) = \sum_{{\text{enhanced states}}~S~{\text{of}}~D} (-1)^{i(S)} q^{j(S)}
\end{equation}
where $i(S)$ $=$ $(w(D) - \sigma(S))/2$ and $j(S)$ $=$ $w(D)$ + $i(S)$ + $\tau(S)$.  Here, $w(D)$ is the writhe number of $D$, which is defined as the number of positive crossings minus negative crossings.  The unnormalized Jones polynomial $\hat{J}(L)(q)$ is $(q + q^{-1}) V_{L}(q)$, with the variable $q$ replaced by $q$ $=$ $- t^{1/2}$ for the well-known (normalized) Jones polynomial $V_{L}(t)$.  Now, we define the Khovanov complex $C^{i, j}(D)$ as the abelian group generated by the enhanced Kauffman states $S$ of a fixed link diagram $D$ satisfying $i(S)$ $=$ $i$ and $j(S)$ $=$ $j$.  Let $T$ be an enhanced state obtained when a neighborhood of only one crossing with a positive marker is replaced by that of the negative marker, where the neighborhood in each of the cases is as listed in Fig. \ref{differential}.  
\begin{figure}
\qquad
\begin{picture}(0,0)
\put(-10,133){(a)}
\put(170,133){(d)}
\put(-10,80){(b)}
\put(170,80){(e)}
\put(-10,28){(c)}
\put(170,28){(f)}
\put(20,33){$1$}
\put(20,82){$1$}
\put(20,133){$x$}
\put(35,152){$S$}
\put(50,133){$1$}
\put(50,82){$x$}
\put(50,33){$1$}
\put(110,133){$x$}
\put(110,82){$x$}
\put(110,33){$1$}
\put(123,152){$T$}
\put(192,17){$1$}
\put(261,17){$1$}
\put(261,43){$x$}
\put(192,68){$1$}
\put(261,68){$x$}
\put(261,95){$1$}
\put(192,120){$x$}
\put(192,168){$S$}
\put(261,147){$x$}
\put(261,120){$x$}
\put(261,168){$T$}
\end{picture}
\includegraphics[width=10cm]{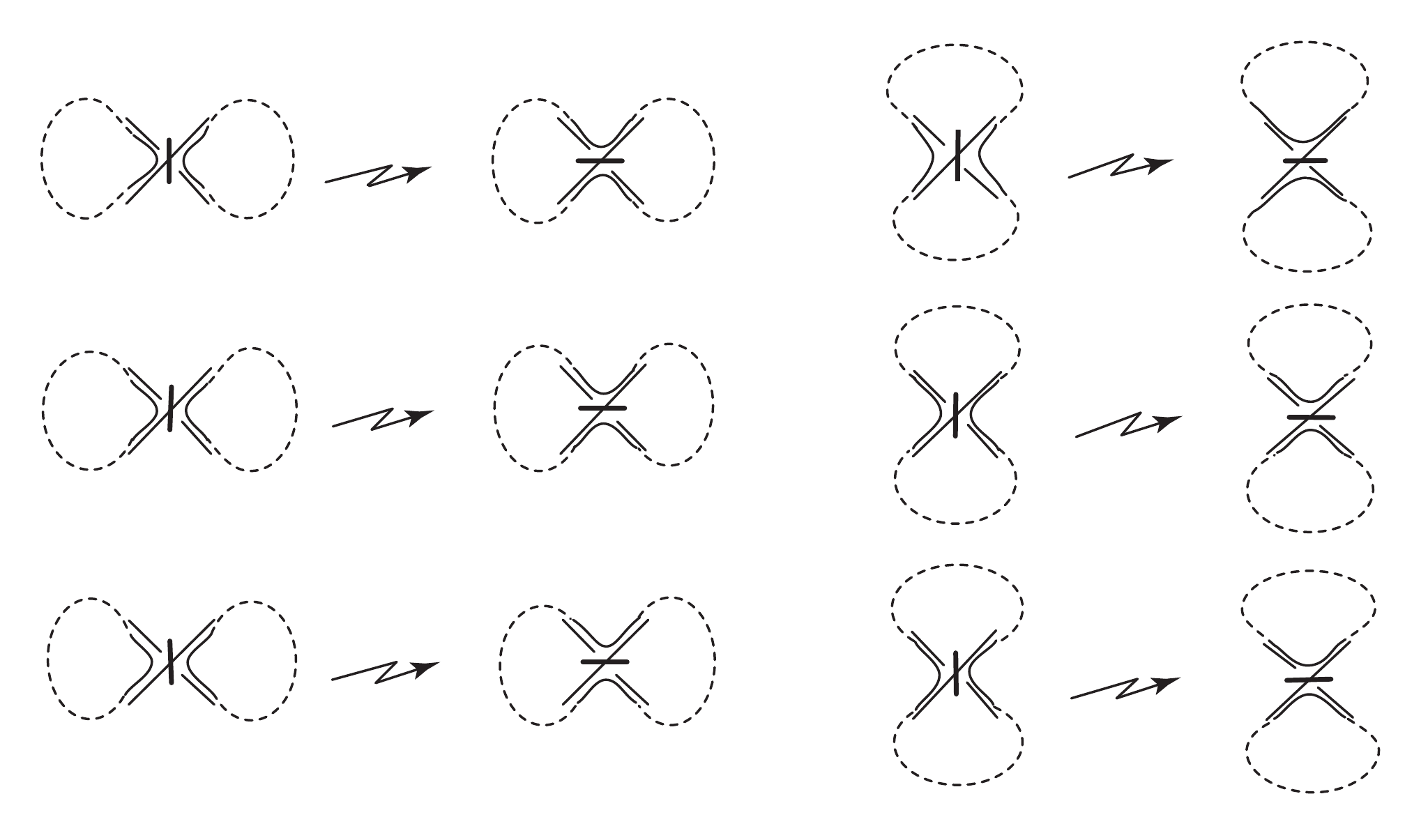}
\caption{Incidence numbers $(S:T)$.  Each $S$ is locally replaced with $T$.  The dotted arcs show how fragments of $S$ or $T$ are connected in the whole $S$ or $T$.  Using another traditional notation, we can write the above formulae as (a) $m(x \otimes 1)$ $=$ $x$, (b) $m(1 \otimes x)$ $=$ $x$, (c) $m(1 \otimes 1)$ $=$ $1$, (d) $\Delta(x)$ $=$ $x \otimes x$, and (e), (f) $\Delta(1)$ $=$ $1 \otimes x$ $+$ $x \otimes 1$.  Here, a circle of enhanced states corresponds to a module $\mathbb{Z}_2 1 \oplus \mathbb{Z}_2 x$ over $\mathbb{Z}_2$.}\label{differential}
\end{figure}
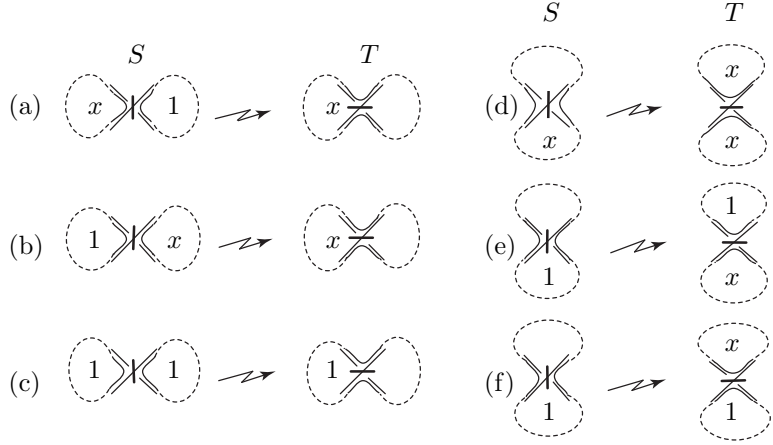
For an arbitrary enhanced state $S$, $d(S)$ is defined by
\begin{equation}
d(S) = \sum_{{\text{enhanced states}}~T} (S : T)\, T
\end{equation}
where the incidence number $(S : T)$ is unity in each of the cases listed in Fig. \ref{differential} and zero if the couple of $S$ and $T$ does not appear in the list of Fig. \ref{differential}.  The map is extended to the homomorphism $d$ from $C^{i, j}(D)$ to $C^{i+1, j}(D)$.  It is a well-known fact that $d$ is the coboundary operator, usually called the {\it{differential}} in the case of Khovanov homology; i.e., $d^{2}$ $=$ $0$.  
\begin{theorem}[Khovanov]
Let $D$ be a diagram of an arbitrary link $L$.  For arbitrary $i$ and $j$, the homology $H^{i}(C^{*, j}(D), d)$ is an isotopy invariant of $L$, and so this homology can be denoted by $H^{i, j}(L)$.  The homology $H^{i, j}(L)$ satisfies
\begin{equation}
\hat{J}(L) = \sum_{j} q^{j} \sum_{i} (-1)^{i} {\rm{rank}} H^{i, j}(L).  
\end{equation}
\end{theorem}

\subsection{Application to Khovanov homology}
Manturov extended the definition of the Khovanov homology to that of virtual knots, denoted here by $KH^{i, j}$ through adding the map between enhanced states of virtual knots.  The problem is that the change of one positive marker to define the differential does not require the change of the component enhanced states for all cases, as shown in Fig. \ref{differential}.  Fortunately, in the case of the coefficient $\mathbb{Z}_{2}$, the definition was extended to virtual knots straightforwardly by regarding these cases as zero maps and using Fig. \ref{differential}.  Moreover, Manturov found the following property \cite{manturov1}: 
\begin{theorem}[Manturov]
For ${KH}^{i, j}(K)$, the Khovanov homology of Manturov, ${KH}^{i, j}(K)$ $\simeq$ ${KH}^{i, j}(p_r(K))$ for an arbitrary virtual knot $K$.  In other words, the Khovanov homology of Manturov is invariant under virtualization of Fig. \ref{virtulization_move}.  
\end{theorem}
\begin{figure}
\begin{center}
\includegraphics[width=8cm]{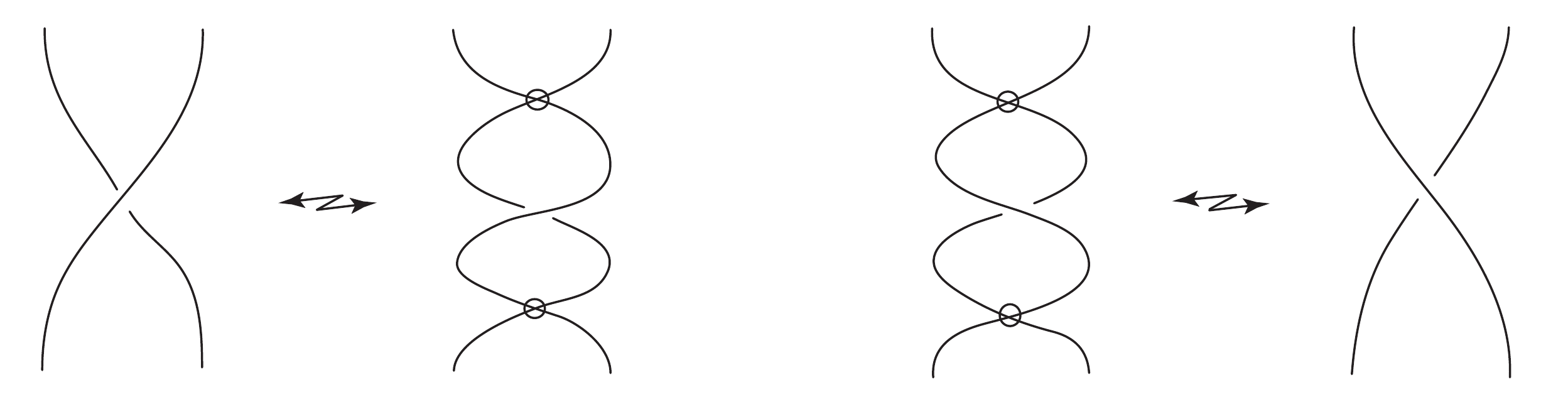}\qquad\qquad
\caption{Virtualization.  Arbitrary orientations of these moves are permitted.}\label{virtulization_move}
\end{center}
\end{figure}
Then, we have the counterpart of Theorem \ref{strongerthan}: 
\begin{theorem}\label{ip}
Let $K$ be an arbitrary long virtual knot.  A pairing of the two Khovanov homologies ${KH}^{i, j}(p_r(K))$ and ${KH}^{i, j}(i(p(K)))$ is stronger than Manturov's Khovanov homology ${KH}^{i, j}(K)$ in terms of an invariant of long virtual knots.  In other words, there exist two long virtual knots $K_1$ and $K_2$ such that ${KH}^{i, j}(p_r(K_1))$ $\simeq$ ${KH}^{i, j}(p_r(K_2))$ for any $(i, j)$ but ${KH}^{i, j}(i(p(K_1)))$ $\not\simeq$ ${KH}^{i, j}(i(p(K_2)))$ for some $(i, j)$.  
\end{theorem}
\begin{proof}
Example \ref{ex_k1_k2} gives what needs to be shown.  
\end{proof}

\begin{ex}\label{ex_k1_k2}
By definition, ${KH}^{i, j}(K_1)$ $\simeq$ ${KH}^{i, j}(K_2)$ for any $(i, j)$.  However, ${KH}^{-2, -5}(i(p(K_1))$ $\simeq$ $\mathbb{Z}_2$ and ${KH}^{-2, -5}(i(p(K_2))$ $\simeq$ $0$. 
\end{ex}

\begin{ex}\label{ex_pra}
Let us consider another type of $p_r$ denoted by $p_{ra}$.  We have ${KH}^{2, -5}$ $(p_r(K_1))$ $\simeq$ ${KH}^{2, -5}(p_r(\emptyset))$ $\simeq$ $0$.  However, ${KH}^{2, -5}(p_{ra}(K_1))$ $\simeq$ $\mathbb{Z}_{2}$, which is not $0$ $\simeq$ $KH^{2, -5}(p_{ra}(\emptyset))$.  
\end{ex}
 
\begin{ex}\label{ia}
Let us consider another type of $i$ denoted by $i_a$.  As described above, ${KH}^{-2, -5}(i(p(K_2)))$ $\simeq$ ${KH}^{-2, -5}(i(p(\emptyset)))$ $\simeq$ $0$.  However, ${KH}^{-2, -5}(i_a(p(K_2)))$ $\simeq$ $\mathbb{Z}_2$, which is not $0$ $\simeq$ ${KH}^{-2, -5}(i_a(p(\emptyset)))$.  
\end{ex}

We also have the counterpart of Theorem \ref{four_thm}: 
\begin{theorem}
All the choices of $p_r$, $p_{ra}$, $p$, $p_a$, $i$, and $i_a$ together generate four types of the Khovanov homology for long virtual knots.  

As a corollary of Theorem \ref{ip}, the tuple of four Khovanov homologies is stronger than the Khovanov homology $KH^{i, j}$ in terms of long virtual knots.  
\end{theorem}

As in Sec. \ref{ver_jones}, four invariants $KH^{i, j}(p_r (D))$, $KH^{i, j}(p_{ra} (D))$, $KH^{i, j}(i \circ p (D))$, and $KH^{i, j}(i_a \circ p (D))$ means considering Khovanov homology for four long virtual knot diagrams $D$, $D^{*}$, $q \circ p(D)$, and $\left( q \circ p(D) \right)^{*}$.

\section*{Acknowledgements}
The author thanks Professor Kouki Taniyama and the referee for useful comments on an earlier version of this paper.  This work was partly supported by Grant-in-Aid for Young Scientists (B) (23740062), IRTG 1529, and a Waseda University Grant for Special Research Projects (Project number: 2010A-863).

\end{document}